\documentclass[12pt,reqno]{amsart}
\usepackage[utf8]{inputenc}
\usepackage{tikz}
\usepackage{amssymb}
\usepackage{amscd}
\usepackage{amsxtra}
\usepackage[mathscr]{eucal}
\usepackage[all]{xy}
\usepackage{amsmath}
\usepackage{csvsimple}
\usepackage{booktabs}
\usepackage{multicol}
\usepackage{url}
\usepackage{listings}
\usetikzlibrary{calc}
\newtheorem{mythm}{Theorem}
\newtheorem{mydef}[mythm]{Definition}
\newtheorem{myprop}[mythm]{Proposition}
\newtheorem{myconj}[mythm]{Conjecture}

\newtheorem{mycor}[mythm]{Corollary}
\newtheorem*{restated-cor}{Corollary \ref{cor:contractible}}
\newtheorem*{restated-thm}{Theorem \ref{thm:embedded-disk}}
\newtheorem*{remark}{Remark}

\usepackage{multirow}
\usepackage{fullpage} 
\usepackage{bbm}
\usepackage{tikz-cd}
\usepackage{hyperref}
\usepackage{mathabx,amsmath,amscd, latexsym, amssymb, bbold}
\usepackage{tcolorbox}
\tcbuselibrary{breakable,xparse,skins}

\definecolor{bg}{gray}{0.95}

\usepackage{enumerate}
\usepackage{amsmath, amsthm} 
\usepackage{amsfonts}
\usepackage{amssymb}
\usepackage{fullpage} 
\usepackage{xcolor}
\usepackage{caption}

\usepackage{mathabx,amsmath,amscd, latexsym, amssymb, bbold}
\usepackage{amsmath, amssymb, tikz}

\definecolor{codegreen}{rgb}{0,0.6,0}
\definecolor{codegray}{rgb}{0.5,0.5,0.5}
\definecolor{codepurple}{rgb}{0.58,0,0.82}
\definecolor{backcolour}{rgb}{0.95,0.95,0.92}

\lstdefinestyle{mystyle}{
    backgroundcolor=\color{backcolour},   
    commentstyle=\color{codegreen},
    keywordstyle=\color{magenta},
    numberstyle=\tiny\color{codegray},
    stringstyle=\color{codepurple},
    basicstyle=\footnotesize\ttfamily,
    breakatwhitespace=false,         
    breaklines=true,                 
    captionpos=b,                    
    keepspaces=true,                 
    numbers=left,                    
    numbersep=5pt,                  
    showspaces=false,                
    showstringspaces=false,
    showtabs=false,                  
    tabsize=2
}

\title{Classification of cellular fake surfaces}

\date{}

\author{Lucas Fagan}
\address{Department of Mathematics, University of California, Santa Barbara, CA 93106, USA}
\email{lucasfagan@math.ucsb.edu}

\author{Yang Qiu}
\address{Chern Institute of Mathematics and LPMC, Nankai University, Tianjin 300071, China}
\email{yangqiu@nankai.edu.cn}

\author{Zhenghan Wang}
\address{Department of Mathematics, University of California, Santa Barbara, CA 93106, USA}
\email{zhenghwa@math.ucsb.edu}

\begin{document}

\begin{abstract}
Generic polyhedra are interesting mathematical objects to study in their own right.  In this paper, we initiate a systematic study of two-dimensional generic polyhedra with an eye towards applications to low-dimensional topology, especially the Andrews-Curtis and Zeeman conjectures.  After recalling the basic notions of generic polyhedra and fake surfaces, we derive some interesting properties of fake surfaces. Our main result is a complete classification of acyclic cellular fake surfaces up to complexity 4 and a classification of acyclic cellular fake surfaces without small disks of complexity 5.  From this classification, we prove the contractibility conjecture for acyclic cellular fake surfaces of complexity 4, and the embedded disk conjecture up to complexity 5. We provide evidence for the conjectures that the probability of being a spine among fake surfaces is 0 and that every contractible fake surface has an embedded disk.
\end{abstract}

\maketitle

\section{Introduction}
The singularities of soap-films in the Plateau problem as well as the skeleton of the microscopic sea life called Radiolaria are examples of generic two-dimensional polyhedra in nature \cite{taylor1976structure}.  Generic polyhedra are interesting mathematical objects to study in their own right. One-dimensional generic polyhedra are simply trivalent graphs.  In this paper, we initiate a systematic study of two-dimensional generic polyhedra with an eye towards applications to low-dimensional topology, especially the Andrews-Curtis conjecture, the Zeeman conjecture, and the 4-dimensional smooth Poincar\'e conjecture \cite{matveev2007algorithmic,hog1993two}.

Two-dimensional generic polyhedra have been used in several different contexts.  They form an essential tool in the study of spines of 3-manifolds \cite{casler1965imbedding,matveev2007algorithmic}, called special simple spines, and quantum invariants of 3-manifolds \cite{turaev1992state}.  Decorated with gleams, they represent 4-manifolds in Turaev's shadow theory \cite{turaev2010quantum, costantino2004short}.  Four manifolds of shadow complexity 0 and 1 are completely classified \cite{martelli2011four,koda2022four}, and Mazur-type contractible manifolds are studied using contractible fake surfaces \cite{hironobu2018mazur,koda2020shadows}. Two-dimensional generic polyhedra also appeared in the capped grope theory as finite stage approximations of embedded disks in 4-manifolds \cite{freedman2014topology}.  But in this paper, we study them as stand-alone topological objects generalizing surfaces.  As such, we are interested in their fundamental properties such as a classification.  We prefer to use the terminology fake surfaces from \cite{ikeda1971acyclic} with a slight modification: our fake surfaces refer to abstract polyhedra with underlying generic 2-polyhedra.  Therefore, our classification is for fake surfaces as singular PL manifolds represented by generic 2-polyhedra.

Arguably the most interesting topological spaces in geometric topology are manifolds, and one important question is their classification.  Manifolds are locally homogeneous in the sense that every point $x$ of an $n$-manifold $M$ has a neighborhood which is homeomorphic to the Euclidean space $\mathbb{R}^n$.  A polyhedron is not locally homogeneous in general, but a generic polyhedron has local regularity.  A polyhedron $P$ is a subset of some Euclidean space $\mathbb{R}^N$ such that every point $x\in P$ has a conical neighborhood $N(x)=xK$ for a compact subset $K$. For a generic $n$-polyhedron $P$, $N(x)$ can be only one of $(n+1)$ cases.

Fake surfaces are generic $2$-polyhedra, so they have only two kinds of singularities beside the regular manifold points as in Fig.~1: a point on an edge where three half planes join (type 1 singularity), and the cone point of a cone on a circle with three radii from the center (type 2 singularity).\footnote{Our fake surfaces are called closed in \cite{ikeda1971acyclic}. Our terminologies are different from the conventions in \cite{ikeda1971acyclic}, where manifold points are called type 1 singularities and our type $k$ singularities are of type $k+1$ there.} 
\begin{figure}[h] 
  \centering
\begin{tikzpicture}
\draw[thick] (0,0)--(2,0);
\draw[dashed] (0,0)--(0.5,0.5);
\draw (0,0)--(0,1);
\draw (0,1)--(2,1);
\draw (0,0)--(-0.5,-0.5);
\draw (-0.5,-0.5)--(1.5,-0.5);
\draw (1.5,-0.5)--(2,0);
\draw[dashed] (0.5,0.5)--(2,0.5);
\draw (2,1)--(2,0);
\draw (2,0.5)--(2.5,0.5);
\draw (2.5,0.5)--(2,0);
\filldraw (1,0) circle(1pt);

\begin{scope}[shift={(4,0)}]
\draw[thick] (0,0)--(2,0);
\draw[dashed] (0,0)--(0.5,0.5);
\draw (0,0)--(0,1);
\draw (0,1)--(2,1);
\draw (0,0)--(-0.5,-0.5);
\draw (-0.5,-0.5)--(1.5,-0.5);
\draw (1.5,-0.5)--(2,0);
\draw[dashed] (0.5,0.5)--(2,0.5);
\draw (2,1)--(2,0);
\draw (2,0.5)--(2.5,0.5);
\draw (2.5,0.5)--(2,0);
\filldraw (1,0) circle(1pt);
\draw[dashed,thick] (1.5,0.5)--(1,0);
\draw[thick] (1,0)--(0.5,-0.5);
\draw[dashed] (1.5,0.5)--(1.5,-0.3);
\draw[dashed] (1.5,-0.3)--(1.3,-0.5);
\draw (1.3,-0.5)--(0.5,-1.3);
\draw (0.5,-0.5)--(0.5,-1.3);
\end{scope}
\end{tikzpicture}
\caption{Two Types of Singularities}
\end{figure}
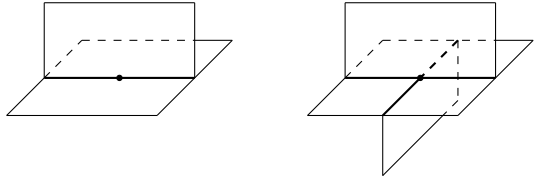

The three kinds of points form an intrinsic stratification of a fake surface: regular manifold points with disk neighborhoods form 2-manifolds, type 1 singularities of the triple edges form 1-manifolds, and type 2 singularities consist of points. In this paper, we will always impose cellularity: the intrinsic stratification leads to a CW complex.  Hence, in a cellular fake surface, every edge is a 1-cell and every face a 2-cell.  The number of vertices of a cellular fake surface $F$ will be called its complexity $c(F)$.  Cellular fake surfaces that are spines of $3$ manifolds are called special simple spines.  Two-dimensional generic polyhedra are special even among generic polyhedra: dimension two is the only dimension in which there are generic polyhedra that are not spines of manifolds of one dimension higher \cite{matveev1973special}.

Cellular fake surfaces have fascinating relations with the 3-dimensional Poincar\'e conjecture, Andrews-Curtis conjecture, and the Zeeman conjecture \cite{matveev2007algorithmic,hog1993two}.  Fake surfaces are studied in \cite{ikeda1971acyclic} with an application to the Zeeman conjecture. 

The main goal of our paper is a classification of low-complexity acyclic cellular fake surfaces. In Conjecture 2 on Page 4 of \cite{robertson1973fake}, it is conjectured that all acyclic fake surfaces of complexity less than 5 are contractible.  We settle this conjecture in the affirmative.  We also found evidence for two new conjectures: the probability of being a spine among fake surfaces is 0, and every contractible fake surface has an embedded disk. Contractible fake surfaces that are spines of the 3-ball could be regarded as interesting shapes of a point.  All physical points occupy non-empty space, and hence have a non-zero size due to the Heisenberg uncertain principle.  If we accept that physical points should have generic shapes, then contractible fake surfaces are attractive candidates.

Our main result is the following classification:

\begin{mythm}
\label{thm:classification}
There are two acyclic cellular fake surfaces of complexity one, 17 of complexity two, 238 of complexity three, and 4618 of complexity four.
    
\end{mythm}

We also have a partial classification of acyclic cellular fake surfaces of complexity 5: we classify those that do not have disks of boundary length 1 and 2, referred to as small disks.

The classification for complexity 1 is in \cite{ikeda1971acyclic}.  Other cases are new.

\begin{mycor}
\label{cor:contractible}
Every acyclic cellular fake surface of complexity less than 5 is contractible.
\end{mycor}

The case of complexity 1, 2 and 3 are in \cite{ikeda1971acyclic,robertson1973fake}. This is not true for complexity 5 as a spine of the Poincar\'e homology sphere shows.

We also prove:
\begin{mythm}
\label{thm:embedded-disk}
Each contractible cellular fake surface of complexity less than 6 has an embedded disk. 
\end{mythm}

The contents of the paper is the following:  in Section~\ref{sec:generic-polyhedra}, we recall the basic notions of generic polyhedra and fake surfaces.  We derive some interesting properties of fake surfaces.  We observe that the analogue of one-dimensional bordism groups is trivial as well. In Section~\ref{sec:classification}, we discuss the classification of acyclic cellular fake surfaces and present the classification of complexity 1 from \cite{ikeda1971acyclic}. In Section~\ref{sec:conjs}, we deduce the contractibility conjecture for complexity 4, and the embedded disk conjecture up to complexity 5 from our classification. 
Appendix~\ref{sec:code} contains the computer code to check that an acyclic cellular fake surface has trivial fundamental group. In Appendix~\ref{sec:fs-c23}, we list all contractible fake surfaces of complexity 2, and give a link to the list of all fake surfaces of complexity 3 and 4, and complexity 5 without small disks. 

\section{Generic polyhedra and singular manifolds}
\label{sec:generic-polyhedra}
A polyhedron $P$ is a subset of some Euclidean space $\mathbb{R}^N$ such that each point $x\in P$ has a conical neighborhood $xK$ for a compact subset $K$, called the link of $x$.  A generic polyhedron $P$ is one such that the links $K$ are of a few special types. 

\subsection{Generic $n$-polyhedra and $n$-singular manifolds}

Let $\Delta^n$ be the standard $n$-simplex with $(n+1)$ vertices and length=1 edges, and $\Delta^n_k$ its $k$-skeleton for $0\leq k\leq n$. Let $\Pi^n_k=\textrm{Cone}(\Delta^{k+1}_{k-1})\times I^{n-k}, 0\leq k\leq n$ be the thickened cone on the codimension 2 skeleton of the $(k+1)$-simplex $\Delta^{k+1}$.  When $k=0$,  $\Pi^n_0$ is simply the $n$-cube $I^n$.

\begin{mydef}
\label{def:generic-polyhedron}
A generic $n$-polyhderon $P$ is a polyhedron such that each point $x\in P$ has a neighborhood that is isomorphic to the $\epsilon$-neighborhood of any point of the $\Pi^n_k$ for some $0<\epsilon <\frac{1}{10},0\leq k\leq n$. 

Points with neighborhoods $\Pi^n_0$ will be called manifold points and are not singularities.  We will call points with neighborhoods $\Pi^n_k, 1\leq k\leq n,$ singularities of type $k$.

A $n$-singular manifold $M$ will be a generic $n$-polyhedron up to PL isomorphism.
\end{mydef}

In our terminology, an $n$-singular manifold have $n$ types of singularities besides the manifold points.

\subsection{Fake surfaces}

\begin{mydef}
\label{def:fake-surface}
A compact 2-singular manifold $F$ will be called a fake surface. The complexity of a fake surface is the number of type 2 singularities.

\end{mydef}

Every fake surface has an intrinsic stratification: the two-dimensional stratum consists of manifold points, the one-dimensional stratum of singularities of type 1, and the points of singularities of type 2.  The number of type 2 singularities of a fake surface is finite.  The type 1 singularities form a 4-regular multigraph, not necessarily connected. A fake surface without any singularities is a surface in the usual sense.

\begin{mydef}
\label{def:cellular-planar}
A fake surface is called cellular if the intrinsic stratification is a CW complex, i.e. all connected open components of the intrinsic strata are disks.

A fake surface is called planar if every edge has at least one type 2 singularity, and every open 2-cell is planar, i.e. of genus 0.  
\end{mydef}
A planar fake surface is not necessarily cellular, as there could be annuli in the 2-skeleton.
\begin{myprop}
\begin{enumerate}
\item A contractible fake surface is cellular.
\item Every cellular fake surface has a connected intrinsic 1-skeleton.
\end{enumerate}
\end{myprop}

Both statements are straightforward.

\subsection{$T$-bundles and spines}

A $T$-bundle can be constructed from the boundary of a 2-cell as shown in \cite{matveev2007algorithmic}. The $T$-bundles can be trivial, a product of $S^1$ with $T$, or nontrivial, a disk attached to the middle circle of a M\"obius band. A nontrivial $T$-bundle in a fake surface is the obstruction to be a spine.

\subsection{Non-spine fake surfaces}

Generic $n$-polyhedra arise as dual celluation of triangulations of $(n+1)$-manifolds and spines of $(n+1)$-manifolds. For $n\geq 3$, every generic $n$-polyhedron is a spine of some manifold \cite{matveev1973special}. The salient feature of dimension two is that most contractible fake surfaces are non-spines, as our classification shows. This is summarized in Table~\ref{tab:my_table}. The numbers of contractible fake surfaces $\{C_n\}$ of complexity $n$ form an interesting sequence, which has not appeared elsewhere \cite{cfs}.  We also observe that as the complexity increases, the fraction of surfaces that are spines decreases. 
\begin{table}[h]
\centering
\captionsetup{justification=centering, width=.8\textwidth}
\caption{Number of Contractible Fake Surfaces by Complexity and Count of Disks with Non-Trivial $T$-Bundles}
\label{tab:my_table} 
\begin{tabular}{c|cccccc|c|c} 
\toprule
\multirow{2}{*}{\textbf{Complexity}} & \textbf{Spines} & \multicolumn{5}{c|}{\textbf{Non-spines}} & \multirow{2}{*}{\textbf{Total}} & \multirow{2}{*}{\textbf{\% Spines}} \\ 
\cmidrule(lr){2-2} \cmidrule(lr){3-7}
 & \textbf{0} & \textbf{1} & \textbf{2} & \textbf{3} & \textbf{4} & \textbf{5} & & \\
\midrule
1	& 1	& 1	& 0		&	&	&	& 2 & 50 \\
2	& 3	& 6	& 6		& 2 	&	&	& 17 & 16.7 \\
3	& 20	& 54	& 89		& 62	& 13	&	& 238 & 8.4 \\
4	& 128	& 607	& 1450	&1533	& 745	& 155	& 4618 & 2.8\\
\bottomrule
\end{tabular}
\end{table}
This provides evidence for the following conjecture:

\begin{myconj}
\label{conj:spines}
The ratio of spine to non-spine contractible fake surfaces of complexity $n$ goes to 0 as the complexity $n$ goes to infinity.
\end{myconj}
   
\subsection{Embedded disks}

\begin{mydef}
\label{def:embedded-disk}
    Given a cellular fake surface, a closed 2-cell in the intrinsic celluation is called an embedded disk if it is an embedded closed disk including all vertices and edges on the boundary.
\end{mydef}

\subsection{Checking For Embedded Disks and $T$-bundles}

In order to check that a 2-cell is embedded, we simply check that each vertex on the boundary of the 2-cell is unique. If a 2-cell is not an embedded disk, the boundary must cross a single vertex twice. 

In order to check if the $T$-bundle of a 2-cell is trivial, we check the behavior along the boundary of the 2-cell. We demonstrate this in the case of an embedded disk, which most lends itself to visualization. Around an embedded disk, it looks as in Figure~\ref{fig:my_hexagon} without loss of generality. (The neighborhood of the embedded disk also contains 2-cells along the vertical lines that are not shown in the figure.) There are two possibilities for the third 2-cell attached along the red/blue edge of the embedded disk: either a disk is attached along the blue line, or along the red line. The former gives us a trivial $T$-bundle, and the latter a nontrivial $T$-bundle. 

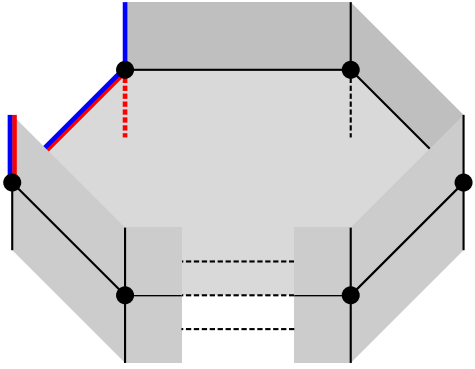
\begin{figure}[h] 
\centering
\begin{tikzpicture}[scale=3, dashed/.style={dash pattern=on 2pt off 1pt}, thick] 

\pgfmathsetmacro{\edgeLength}{0.3} 

\fill[lightgray] (0.5,0.5-\edgeLength) -- (0.5,0.5) -- (1.5,0.5) -- (1.5,0.5-\edgeLength) -- cycle; 
\fill[lightgray] (1.5,0.5-\edgeLength) -- (1.5,0.5) -- (2,0) -- (2,-\edgeLength) -- cycle; 
\fill[lightgray!80] (0,0-\edgeLength) -- (0.5,-0.5-\edgeLength) -- (0.5,-0.5) -- (0,0) -- cycle; 
\fill[lightgray!80] (0.75,-0.5-\edgeLength) -- (0.5,-0.5-\edgeLength) -- (0.5,-0.5) -- (0.75,-0.5) -- cycle;  
\fill[lightgray!80] (2,0-\edgeLength) -- (1.5,-0.5-\edgeLength) -- (1.5,-0.5) -- (2,0) -- cycle; 
\fill[lightgray!80] (1.25,-0.5-\edgeLength) -- (1.5,-0.5-\edgeLength) -- (1.5,-0.5) -- (1.25,-0.5) -- cycle;

\fill[lightgray] (0.5,0.5) -- (0.5,0.5+\edgeLength) -- (1.5,0.5+\edgeLength) -- (1.5,0.5) -- cycle; 
\fill[lightgray] (1.5,0.5+\edgeLength) -- (1.5,0.5) -- (2,0) -- (2,\edgeLength) -- cycle; 
\fill[lightgray!60, opacity=0.92] (0,0) -- (0.5,0.5) -- (1.5,0.5) -- (2,0) -- (1.5,-0.5) -- (0.5,-0.5) -- cycle;

\draw (0,0) -- (0.5,0.5) -- (1.5,0.5) -- (2,0) -- (1.5,-0.5); 
\draw[dashed] (1.25,-0.5) -- (0.75,-0.5);
\draw (1.5,-0.5) -- (1.25,-0.5);
\draw (0.75,-0.5) -- (0.5,-0.5);
\draw[dashed] (1.25,-0.5+\edgeLength/2) -- (0.75,-0.5+\edgeLength/2);
\draw[dashed] (1.25,-0.5-\edgeLength/2) -- (0.75,-0.5-\edgeLength/2);

\draw[red,shift={(0.01,0)},line width=1.75pt] (0,0) -- (0.5,0.5);
\draw[blue,shift={(-0.01,0)},line width=1.75pt] (0,0) -- (0.5,0.5);

\draw[red,dashed,line width=1.75pt] (0.5,0.5) -- (0.5,0.5-\edgeLength);
\draw[blue,line width=1.75pt] (0.5,0.5) -- (0.5,0.5+\edgeLength);

\fill[lightgray!80] (0,0) -- (0.5,-0.5) -- (0.5,-0.5+\edgeLength) -- (0,0+\edgeLength) -- cycle;  
\fill[lightgray!80] (0.75,-0.5) -- (0.5,-0.5) -- (0.5,-0.5+\edgeLength) -- (0.75,-0.5+\edgeLength) -- cycle;  
\fill[lightgray!80] (2,0) -- (1.5,-0.5) -- (1.5,-0.5+\edgeLength) -- (2,0+\edgeLength) -- cycle;  
\fill[lightgray!80] (1.25,-0.5) -- (1.5,-0.5) -- (1.5,-0.5+\edgeLength) -- (1.25,-0.5+\edgeLength) -- cycle;  
\draw (0.5,-0.5) -- (0,0); 
\draw (2,0) -- (1.5,-0.5);

\foreach \x/\y in {2/0, 1.5/-0.5, 0.5/-0.5} { 
  \filldraw (\x,\y) circle (1pt); 
  \draw (\x,\y) -- (\x, \y + \edgeLength); 
  \draw (\x,\y) -- (\x, \y - \edgeLength); 
}
\draw (0,0) -- (0,0-\edgeLength);
\draw[red,shift={(0.01,0)},line width=1.75pt] (0,0) -- (0,0+\edgeLength);
\draw[blue,shift={(-0.01,0)},line width=1.75pt] (0,0) -- (0,0+\edgeLength);

\filldraw (0.5,0.5) circle (1pt);
\filldraw (0,0) circle (1pt);
\foreach \x/\y in {1.5/0.5} { 
  \filldraw (\x,\y) circle (1pt); 
  \draw (\x,\y) -- (\x, \y + \edgeLength); 
  \draw[dashed] (\x,\y) -- (\x, \y - \edgeLength); 
}
\end{tikzpicture}
\caption{Part of a Neighborhood of an Embedded Disk in a Fake Surface}
\label{fig:my_hexagon}
\end{figure}

\subsection{Algorithms for deforming 2-complexes into fake surfaces}

Given a 2-complex $K$, there are several algorithms in the literature to change $K$ into a fake surface $P_K$ by $3$-deformations.  Many are based on the so-called banana-pineapple trick from Remark 2, Page 26 of \cite{zeeman1963seminar} for spines of 3-manifolds: ``For choose a spine in the interior; expand each edge like a banana and collapse from one side; then expand each vertex like a pineapple and collapse from one face."  There is also a separate paper \cite{wright1973formal}.  We will make an estimate of the complexity of the resulting fake surface following the banana-pineapple trick in the case of the standard 2-complex realization of a group. 

\begin{myprop}
\label{prop:complexity}
Let $G=\langle x_1,\ldots,x_k \mid r_1, \ldots, r_\ell \rangle$ be a group presentation. Let $n_{x_i}$ be the number of occurrences of $x_i$ across all the relators. Let $L=\sum_i n_{x_i}$ and without loss of generality assume $n_{x_1}\geq n_{x_i}$ for all $i$. Then applying banana-pineapple to the standard 2-complex realization of $G$ gives a fake surface of complexity $2L-4k-n_{x_1}+2$.
\end{myprop}
\begin{proof}
    Consider the edge corresponding to $x_i$ in the standard 2-complex realization of $G$. Thickening it to a banana and collapsing through a longitudinal piece of the peel creates $n_{x_i}-2$ new edges, as each of the $n_{x_i}$ two-cells attached to that edge create a new edge except for the two that border the collapsed longitudinal piece of the peel. Now thickening the vertex and collapsing the resulting pineapple through a piece of the peel creates a new vertex for each of the $\sum_{i=1}^k 2(n_{x_i}-2)$ edges (where we multiply by 2 because each edge appears twice in a neighborhood of the vertex), except those that border the piece of the peel from which we collapse. In order to make the complexity as small as possible, we choose to collapse from the piece of the peel in the middle of the $n_{x_1}-2$ edges coming from one side of the original $x_1$ edge, thus giving $\sum_{i=1}^k 2(n_{x_i}-2)-(n_{x_1}-2)$ vertices. Simplifying gives $2L-4k-n_{x_1}+2$. 
\end{proof}

\subsection{Mapping class groups of $2$-singular manifolds}

Given a cellular fake surface $F$, the group of PL homeomorphims up to isotopy will be called the mapping class group of $F$.  Unlike the mapping class group of surfaces, they are always finite groups.

\subsection{Bordism groups of $2$-singular manifolds}

The notion of $n$-singular manifolds with boundaries can also be defined generalizing the case of fake surfaces in \cite{ikeda1971acyclic}. It is natural to ask what the analogous bordism groups are. 
\begin{mydef}
\label{def:fake-surface-with-boundary}
A fake surface with boundary is a 2-polyhedron that could contain points as shown in Fig.~\ref{fig:a} in addition to the ones in Definition \ref{def:fake-surface}.

\end{mydef}

We prove the simplest case that the bordism group of $1$-singular manifolds is trivial, i.e. every $1$-singular manifold bounds a fake surface with boundary.

\begin{figure}[h] 
  \centering
\begin{tikzpicture}
\begin{scope}[shift={(-3,0)}]

\draw[thick] (1.5,0.5)--(0.5,-0.5);
\draw (0.5,-0.5)--(1.5,-0.5);
\draw (1.5,-0.5)--(2.5,0.5);
\draw (1.5,0.5)--(2.5,0.5);

\filldraw (1,0) circle(1pt);

\end{scope}

\draw (1,0)--(2,0);
\draw[dashed,thick] (1,0)--(1.5,0.5);
\draw[thick] (1,0)--(1,1);
\draw (1,1)--(2,1);
\draw[thick] (1,0)--(0.5,-0.5);
\draw (0.5,-0.5)--(1.5,-0.5);
\draw (1.5,-0.5)--(2,0);
\draw[dashed] (1.5,0.5)--(2,0.5);
\draw (2,1)--(2,0);
\draw (2,0.5)--(2.5,0.5);
\draw (2.5,0.5)--(2,0);
\filldraw (1,0) circle(1pt);

\begin{scope}[shift={(4,0)}]
\draw (0,0)--(1,0);
\draw[dashed] (0,0)--(0.5,0.5);
\draw (0,0)--(0,1);
\draw (0,1)--(1,1);
\draw (0,0)--(-0.5,-0.5);
\draw (-0.5,-0.5)--(1.5,-0.5);
\draw (1.5,-0.5)--(2,0);
\draw[dashed] (0.5,0.5)--(1,0.5);
\draw (1,0.5)--(2.5,0.5);
\draw (2.5,0.5)--(2,0);
\draw[thick] (1,1)--(1,0);
\filldraw (1,0) circle(1pt);

\end{scope}
\end{tikzpicture}
\caption{Boundary points}
    \label{fig:a}

\end{figure}

\begin{myprop}
Any finite trivalent graph is the boundary of a fake surface.
\end{myprop}
\begin{proof}
Given any finite trivalent graph $T$ in $\mathbb{R}^3$, we can construct a fake surface whose boundary is $T$ as follows. First we take the one-point cone $C$ over $T$ and denote the cone point by $P$. Define an immersion $i:C\longrightarrow\mathbb{R}^3$ as follows. Since $T$ is finite, after appropriate isotopy we can project $T$ to one nearby plane, such that no two vertices are projected to the same point, any vertex is not projected to the interior of any edge, the projection of any edge does not intersect itself in the interior, and the projections of any two edges either intersect each other transversely or do not intersect. Choose one point $Q$ far from the plane. Connect points on $T$ and $Q$ by lines as shown in Fig.~\ref{fig:enter-label}. This gives an immersion $i:C\longrightarrow\mathbb{R}^3$ where $i(P)=Q$. We get that the cone of $Q$ over an edge of $T$ does not intersect itself and the intersections of two such cones may be some arcs and circles. Suppose that $B$ is a 3-ball containing $Q$ in $\mathbb{R}^3$ such that $B$ does not intersect $T$ and let $S$ be the boundary of $B$. Then the intersections between $i(C)$ and $S$ give a graph $R$ which contains only trivalent vertices and 4-valent vertices as shown by the red graph in Fig.~\ref{fig:enter-label}. Each trivalent vertex is the intersection between an arc from one vertex on $T$ to $Q$ with $S$. Each 4-valent vertex is given by a blue arc that goes towards $Q$ and is the intersection of the images under $i$ of two 2-faces of $C$. Suppose the 2-polyhedron $F$ is obtained from $C$ by removing the preimage of the part of $i(C)$ in $S$ under $i$. Then $F\cup_R S$ gives a fake surface with $T$ as its boundary. The neighborhood of a trivalent vertex is the union of an I-bundle of Y and a disk which gives a vertex of a fake surface. Near a 4-valent vertex where two 2-faces of $C$ intersect, we can imagine one 2-face on one side of $S$ and the other 2-face on the other side which gives a vertex of a fake surface.
\end{proof}
\begin{figure}
    \centering
\begin{tikzpicture}
\draw (0,0)--(0,2);
\draw (0,0)--(-1,-1);
\draw (0,0)--(1,-1);
\draw[red] (5.5,0)--(5.5,1);
\draw[red] (5.5,0)--(5,-0.5);
\draw[red] (5.5,0)--(6,-0.5);
\draw (-0.8,0.2)--(-0.2,0.8);
\draw (0.2,1.2)--(1.2,2.2);
\draw[red] (5.1,0.1)--(5.9,0.9);
\draw (0,0)--(5.5,0);
\draw[dashed] (5.5,0)--(6,0);
\draw (0,2)..controls(3,2)..(5.5,1);
\draw (-0.8,0.2)--(5.1,0.1);
\draw (1.2,2.2)..controls(3.7,2.2)..(5.9,0.9);
\draw[dashed] (5.1,0.1)..controls(5.5,0.2)..(6,0);
\draw[dashed] (5.9,0.9)--(6,0);
\draw[dashed,blue] (5.5,1)--(6,0);
\draw[dashed] (5.5,0.5)--(6,0);
\draw[blue] (0,1)--(5.5,0.5);

\draw (4,-0.3) arc(-170:110:2);
\filldraw (5.5,0) circle(1pt);
\filldraw (5.5,0.5) circle(1pt);
\filldraw (0,0) circle(1pt);
\filldraw (6,0) circle(1pt);

\node[right] at (6,0){$i(P)=Q$};
\node at (0,-2){$T$};
\node at (6,-3){$S$};
\node at (5.5,-1){$R$};
\end{tikzpicture}
    \caption{Construction of a fake surface}
    \label{fig:enter-label}
\end{figure}

\section{Classification of Acyclic Cellular Fake surfaces}
\label{sec:classification}

In this section, we discuss the classification of acyclic cellular fake surfaces of complexities 1, 2, 3, and 4, along with complexity 5 in the case of no disks of boundary length 1 and 2, which we henceforth call small disks. Full data for the classification can be found on GitHub at \url{https://github.com/lucasfagan/Fake-Surfaces}. Acyclic cellular fake surfaces of complexity 1 are classified in \cite{ikeda1971acyclic}, and there are two of them: one spine of the three ball, the abalone, and one non-spine.

\subsection{Notation and convention}
\label{subsec:notation-convention}
We present a fake surface $F$ as $(\Gamma_i^j, M)$ where $\Gamma_i^j$ is the adjacency matrix of the $i$th 1-skeleton of complexity $j$, and $M$ is a matrix representing the 2-cell attaching maps. A negative entry denotes going along the edge in the opposite way. In this section, we pad the attaching maps with 0s when they are shorter than the longest disk. 

The last two columns are the answers to whether the disk is embedded and has a trivial $T$-bundle, respectively.  

Our notation of fake surfaces relies on ordering the vertices and edges. We adopt the following conventions.

We present the 1-skeleta, viewed as 4-regular multigraphs, as follows: consider the map
$\mathsf{D}:M_{n}(\mathbb{Z})\to \mathbb{Z}$ that gives a decimal representation of a matrix by simply concatenating the rows from top to bottom into an integer with $n^2$ digits. An example makes this clear: 
\[ \mathsf{D}\left(\begin{bmatrix}1&2\\3&4\end{bmatrix}\right)=1234.
\]
Explicitly, if $A=(a_{i,j}),  \mathsf{D}(A)=\sum_{i, j=1}^n 10^{n(n-i)+n-j} a_{i,j}$ as an integer.

For a given 1-skeleton of complexity $n$, we choose the adjacency matrix $A=(a_{i,j})$ and therefore an ordering of the vertices that maximizes $\mathsf{D}(A)$. Note this also gives a canonical ordering of the 1-skeleta for a given complexity: from largest value of $\mathsf{D}(A)$ to smallest. 

We then label the edges in the order they appear in the upper-right triangle of the chosen adjacency matrix, read left-to-right and top-to-bottom. To be precise, using our ordering of the vertices, let $(i_1,j_1)$ and $(i_2,j_2)$ be edges satisfying $i_1\leq j_1$ and $i_2\leq j_2$. Then $(i_1,j_1)<(i_2,j_2)$ if $i_1<i_2$ or $i_1=i_2$ and $j_1<j_2$. We then label edges from 1 to $2n$ in accordance with this ordering. The order of multiple edges between two vertices can be chosen arbitrarily. 

We provide an example to make this clear. Consider the graph $\Gamma$ given by the adjacency matrix 
\[A=
\begin{bmatrix}
2&1&1\\
1&0&3\\
1&3&0
\end{bmatrix}.
\]
Observe that this ordering of the vertices is the one that maximizes $\mathsf{D}(A)$. Then the self-loop is first edge, the two other edges at that vertex are the next two edges, and the three edges connecting the other two vertices are the fourth, fifth, and sixth edges. Thus we label the edges by any of the 12 valid labelings. One such labeling is shown in Fig.~\ref{fig:example}:
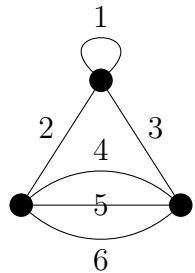
\begin{figure}[h]
\caption{Valid Edge Labeling for $\Gamma$}
\label{fig:example}
\centering
\begin{tikzpicture}[
    node distance=1.5cm, 
    every node/.style={circle, draw, fill=black, inner sep=3pt},
    edge_label/.style={midway, draw=none, inner sep=2pt, fill=none} 
] 
    \node (A) {};
    \node (B) [below right of=A, yshift=-0.6cm] {};
    \node (C) [below left of=A, yshift=-0.6cm] {};

    \draw (A) to node[edge_label, above right] {3} (B);
    \draw (B) to[bend left=40] node[edge_label, below] {6} (C); 
    \draw (B) to node[edge_label] {5} (C); 
    \draw (B) to[bend right=40] node[edge_label, above] {4} (C);
    \draw (A) to node[edge_label, above left] {2} (C); 
    \draw (A) to [out=135, in=45, looseness=10] node[edge_label, above] {1} (A); 
\end{tikzpicture}
\end{figure}

Recall that a fake surface is acyclic if it is homologous to a point, and contractible if it is homotopic to a point.  
Any acyclic cellular fake surface with connected 1-skeleton and $t$ many type 2 singularities has $t+1$ disks and $2t$ edges.

\subsection{Classification strategy}

\subsubsection{Classification of 4-regular multigraphs}
We first enumerate all connected 4-regular multigraphs viewed as the 1-skeleta. This can be easily done through enumerating valid adjacency matrices. We provide some combinatorics of the 1-skeleta for each complexity in Tables~\ref{tab:complexity_table} and \ref{tab:cycle_table}.
\begin{table}[h]
\centering
\caption{One-Skeleta by Complexity and Number of Self-loops}
\label{tab:complexity_table}
\begin{tabular}{c|cccccccc|c} 
\toprule
\textbf{Complexity} & \textbf{0} & \textbf{1} & \textbf{2} & \textbf{3} & \textbf{4} & \textbf{5} & \textbf{6} & \textbf{7} & \textbf{Total} \\
\midrule
1 & 0 & 0 & 1 & & & & & & 1 \\
2 & 1 & 0 & 1 & & & & & & 2 \\
3 & 1 & 1 & 1 & 1 & & & & & 4 \\
4 & 3 & 2 & 3 & 1 & 1 & & & & 10 \\
5 & 6 & 7 & 7 & 5 & 2 & 1 & & & 28 \\
6 & 19 & 21 & 28 & 16 & 10 & 2 & 1 & & 97 \\
7 & 50 & 85 & 98 & 72 & 36 & 14 & 3 & 1 & 359\\
\bottomrule
\end{tabular}
\end{table}
\begin{table}[h]
\centering
\caption{One-skeleta by Length of Shortest Cycle}
\label{tab:cycle_table}
\begin{tabular}{c|ccc|c} 
\toprule
\textbf{Complexity} & \textbf{1} & \textbf{2} & \textbf{3} & \textbf{Total} \\
\midrule
1 & 1 & 0 & 0 & 1 \\
2 & 1 & 1 & 0 & 2 \\
3 & 3 & 1 & 0 & 4 \\
4 & 7 & 3 & 0 & 10 \\
5 & 22 & 5 & 1 & 28 \\
6 & 78 & 18 & 1 & 97 \\
7 & 309 & 48 & 2 & 359\\
\bottomrule
\end{tabular}
\end{table}

\subsubsection{Creating Exhaustive List of Fake Surfaces}
We classify fake surfaces for specific 1-skeleta. Given a 1-skeleton with oriented, labeled edges, we first enumerate all ways to attach oriented disks to the 1-skeleton within the constrains of having a type 2 singularity at each vertex. Then, we enumerate all combinations of disks which may form a fake surface and check whether or not there is a type 2 singularity at each vertex. If this is satisfied, we check that the surface is acyclic by collapsing a maximal tree of the 1-skeleton and checking that the determinant of the boundary map $\mathbb{Z}^{n+1}\to \mathbb{Z}^{2n}$ is $\pm 1$. Then we store these acyclic fake surfaces.

\subsubsection{Removing Redundancy}
The final step is to remove redundancy from our list of fake surfaces. We classify fake surfaces up to PL homeomorphism, which is the same as the equivalence relation defined by sequences of the following combinatorial conditions:
\begin{enumerate}
    \item Relabel the edges of the graph (i.e., apply a graph automorphism)
    \item Change the orientation of an edge (i.e., invert all appearances of that edge)
    \item Change the orientation of a disk (i.e., reverse the order of the edges and invert each edge) 
    \item Cyclically rotate the list of edges of a disk
\end{enumerate}

\begin{remark}
Two cellular fake surfaces are connected by a finite sequence of the above moves if and only if they are PL homeomorphic to each other, i.e., after endowing two surfaces with triangulations, the two simplicial complexes are simplicial isomorphic after subdivision \cite{rourke1982introduction}. A sequence of the moves gives a homeomorphism between two fake surfaces which maps vertices to vertices, edges to edges, and regions to regions. Conversely, given triangulations on two fake surfaces respectively, we can subdivide their 1-skeletons such that they have the same simplicial structures and extend them to regions which are open disks.
\end{remark}

We then output the list of unique surfaces up to these aforementioned rules. 
The code for this classification can be found at \url{https://github.com/lucasfagan/Fake-Surfaces}.

\subsection{Fake Surfaces up to Complexity 5}
\subsubsection{Complexity 1}
There is a unique 1-skeleton for complexity 1, which is given by $\mathbf{\Gamma^1}=
(\Gamma_1^1)=(\begin{bmatrix}
4
\end{bmatrix})$ and shown in Fig.~\ref{fig:c1}.
\begin{figure}[h]
\centering
\begin{tikzpicture}[node distance=2cm, every node/.style={circle, draw, fill=black, inner sep=3pt}] 
  \node (A) {};
  \draw (A) to [out=135, in=225, looseness=15] (A);
  \draw (A) to [out=45, in=-45, looseness=15] (A);
\end{tikzpicture}
\caption{One-Skeleton for Complexity 1}
\label{fig:c1}
\end{figure}
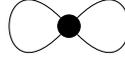

There are two acyclic fake surfaces of complexity 1, first classified in \cite{ikeda1971acyclic}, which we reproduce here in the notation introduced in Section~\ref{subsec:notation-convention}.
\[\left(\Gamma_1^1,\left[\begin{array}{ccccc|cc}
1&2&2&-1&-2&N&N\\
1&0&0&0&0&Y&Y\end{array}\right]\right)\]
\[\left(\Gamma_1^1,\left[\begin{array}{ccccc|cc}
1&2&2&1&-2&N&Y\\
1&0&0&0&0&Y&Y\end{array}\right]\right)\]

The first is a non-spine as the top disk has a non-trivial $T$-bundle, and the second is a spine of the 3-ball called the abalone. 

Note that around a vertex with no self-loops, edge orientations are unnecessary as they are implied by the ordering, but these two surfaces demonstrate that edge orientations for self-loops is critical. Indeed, describing attaching maps with unoriented edges fail to differentiate between these two surfaces.  
\subsubsection{Complexity 2}
For complexity 2, there are two 1-skeleta, ordered according to our convention in the following way and shown in Fig.~\ref{fig:c2}. \[\mathbf{\Gamma^2}=
(\Gamma_1^2,\Gamma_2^2)=\left(\begin{bmatrix}
2&2\\
2&2\\
\end{bmatrix},
\begin{bmatrix}
0&4\\
4&0\\
\end{bmatrix}
\right)
\]
\begin{figure}[h]
\centering
\begin{tikzpicture}[node distance=2cm, every node/.style={circle, draw, fill=black, inner sep=3pt}] 

  \def\hs{4} 

  \begin{scope}[shift={(0,0)}] 
    \node (A) {};
    \node (B) [right of=A] {};

    \draw (A) to[bend left] (B);
    \draw (A) to[bend right] (B);

    \draw (A) to [out=135, in=225, looseness=15] (A);
    \draw (B) to [out=45, in=-45, looseness=15] (B);

    \node[opacity=0] at (0.5,0.8) {}; 
    \node[opacity=0] at (0.5,-0.8) {}; 
  \end{scope}

  \begin{scope}[shift={(\hs,0)}]  
    \node (C) {};
    \node (D) [right of=C] {};

    \draw (C) to[bend left=60] (D);
    \draw (C) to[bend left=20] (D);
    \draw (C) to[bend right=20] (D);
    \draw (C) to[bend right=60] (D);

    \node[opacity=0] at (0.5,0.8) {};
    \node[opacity=0] at (0.5,-0.8) {};
  \end{scope}

\end{tikzpicture}
\caption{One-Skeleta for Complexity 2}
\label{fig:c2}
\end{figure}
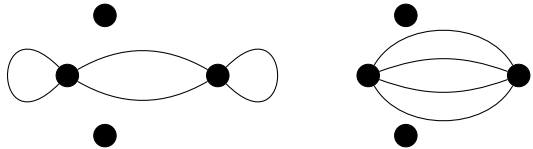

One well-known example of a complexity 2 fake surface is Bing's house, also called the house with two rooms. This surface can be presented as follows:
\[\left(\Gamma_1^2,\left[\begin{array}{cccccccccc|cc}
3&-1&-3&2&-1&-2&-4&2&-3&-4&N&Y\\
4&0&0&0&0&0&0&0&0&0&Y&Y\\
-1&0&0&0&0&0&0&0&0&0&Y&Y\end{array}\right]\right).\]
There is another spine of the 3-ball, called the ``mutant'' of the house with two rooms in \cite{matveev2007algorithmic}, which demonstrates again the need for edge orientations:
\[\left(\Gamma_1^2,\left[\begin{array}{cccccccccc|cc}
3&-1&-3&2&1&-2&4&2&-3&-4&N&Y\\
4&0&0&0&0&0&0&0&0&0&Y&Y\\
-1&0&0&0&0&0&0&0&0&0&Y&Y\end{array}\right]\right).\] 
Observe that a description without edge orientations would see these two surfaces as identical.  

There are 17 complexity-2 acyclic cellular fake surfaces, 15 with the first 1-skeleton and two with the second. The remaining acyclic cellular fake surfaces of complexity 2 can be found in Appendix~\ref{sec:fs-c23}.
\subsubsection{Complexity 3}
For complexity 3, there are four 1-skeleta, presented and ordered according to our convention in the following way and shown in Fig.~\ref{fig:c3}. 
\[
\mathbf{\Gamma^3}=(\Gamma_1^3,\Gamma_2^3,\Gamma_3^3,\Gamma_4^3)=
\left(
\begin{bmatrix}
2&2&0\\
2&0&2\\
0&2&2
\end{bmatrix},
\begin{bmatrix}
2&1&1\\
1&2&1\\
1&1&2
\end{bmatrix},
\begin{bmatrix}
2&1&1\\
1&0&3\\
1&3&0
\end{bmatrix},
\begin{bmatrix}
0&2&2\\
2&0&2\\
2&2&0
\end{bmatrix}
\right)
\]
\begin{figure}[h]
\centering
\begin{tikzpicture}[node distance=1.5cm, every node/.style={circle, draw, fill=black, inner sep=3pt}] 

  \def\hs{3.5} 

  \begin{scope}[shift={(0,0)}] 
    \node (A) {};
    \node (B) [below right of=A, yshift=-0.6cm] {};
    \node (C) [below left of=A, yshift=-0.6cm] {};

    \draw (A) to[bend left=20] (B);
    \draw (A) to[bend right=20] (B);
    \draw (B) to[bend left=20] (C);
    \draw (B) to[bend right=20] (C);
    \draw (A) to [out=135, in=225, looseness=10] (A);
    \draw (C) to [out=135, in=225, looseness=10] (C);
  \end{scope}

  \begin{scope}[shift={(\hs,0)}] 
    \node (A) {};
    \node (B) [below right of=A, yshift=-0.6cm] {};
    \node (C) [below left of=A, yshift=-0.6cm] {};

    \draw (A) to (B);
    \draw (B) to (C);
    \draw (A) to (C);
    \draw (A) to [out=135, in=45, looseness=10] (A);
    \draw (C) to [out=135, in=225, looseness=10] (C);
    \draw (B) to [out=45, in=-45, looseness=10] (B);
  \end{scope}

  \begin{scope}[shift={(2*\hs,0)}] 
    \node (A) {};
    \node (B) [below right of=A, yshift=-0.6cm] {};
    \node (C) [below left of=A, yshift=-0.6cm] {};

    \draw (A) to (B);
    \draw (B) to[bend left=20] (C);
    \draw (B) to (C);
    \draw (B) to[bend right=20] (C);
    \draw (A) to (C);
    \draw (A) to [out=135, in=45, looseness=10] (A); 
  \end{scope}

  \begin{scope}[shift={(3*\hs,0)}] 
    \node (A) {};
    \node (B) [below right of=A, yshift=-0.6cm] {};
    \node (C) [below left of=A, yshift=-0.6cm] {};

    \draw (A) to[bend left=20] (B);
    \draw (A) to[bend right=20] (B);
    \draw (B) to[bend left=20] (C);
    \draw (B) to[bend right=20] (C);
    \draw (A) to[bend left=20] (C);
    \draw (A) to[bend right=20] (C);
  \end{scope}

\end{tikzpicture}
\caption{One-Skeleta for Complexity 3}
\label{fig:c3}
\end{figure}
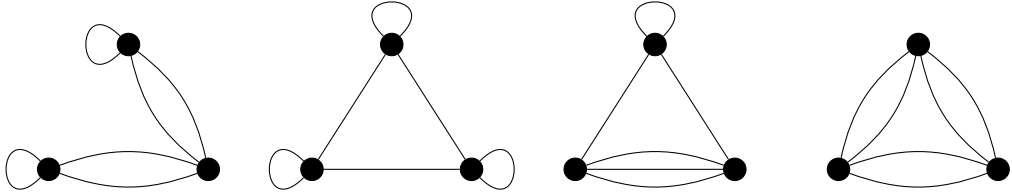

There are 238 acyclic cellular fake surfaces of complexity 3, with 14, 65, 97 and 62 coming from the first, second, third, and fourth one-skeleton, respectively. All are listed at \url{https://github.com/lucasfagan/Fake-Surfaces}.

\subsubsection{Complexity 4}
For complexity 4, there are 10 one-skeleta, presented and ordered according to our convention in the following way and shown in Fig.~\ref{fig:c4}.
\begin{align*}
\mathbf{\Gamma^4}=
& \left(
\begin{bmatrix}
2 & 2 & 0 & 0 \\
2 & 0 & 2 & 0 \\
0 & 2 & 0 & 2 \\
0 & 0 & 2 & 2
\end{bmatrix}, 
\begin{bmatrix}
2 & 2 & 0 & 0 \\
2 & 0 & 1 & 1 \\
0 & 1 & 2 & 1 \\
0 & 1 & 1 & 2
\end{bmatrix},
\begin{bmatrix}
2 & 2 & 0 & 0 \\
2 & 0 & 1 & 1 \\
0 & 1 & 0 & 3 \\
0 & 1 & 3 & 0
\end{bmatrix},
\begin{bmatrix}
2 & 1 & 1 & 0 \\
1 & 2 & 0 & 1 \\
1 & 0 & 2 & 1 \\
0 & 1 & 1 & 2
\end{bmatrix},
\begin{bmatrix}
2 & 1 & 1 & 0 \\
1 & 2 & 0 & 1 \\
1 & 0 & 0 & 3 \\
0 & 1 & 3 & 0
\end{bmatrix},
\right. \\
& \left.
\begin{bmatrix}
2 & 1 & 1 & 0 \\
1 & 0 & 2 & 1 \\
1 & 2 & 0 & 1 \\
0 & 1 & 1 & 2
\end{bmatrix},
\begin{bmatrix}
2 & 1 & 1 & 0 \\
1 & 0 & 1 & 2 \\
1 & 1 & 0 & 2 \\
0 & 2 & 2 & 0
\end{bmatrix},
\begin{bmatrix}
0 & 3 & 1 & 0 \\
3 & 0 & 0 & 1 \\
1 & 0 & 0 & 3 \\
0 & 1 & 3 & 0
\end{bmatrix},
\begin{bmatrix}
0 & 2 & 2 & 0 \\
2 & 0 & 0 & 2 \\
2 & 0 & 0 & 2 \\
0 & 2 & 2 & 0 
\end{bmatrix},
\begin{bmatrix}
0 & 2 & 1 & 1 \\
2 & 0 & 1 & 1 \\
1 & 1 & 0 & 2 \\
1 & 1 & 2 & 0
\end{bmatrix}
\right)
\end{align*}
\begin{figure}[h]
\centering
\begin{tikzpicture}[node distance=1.5cm, every node/.style={circle, draw, fill=black, inner sep=3pt}] 

  \def\hs{3} 
  \def\vs{3} 

  \begin{scope}[shift={(0,0)}] 
    \node (A) {};
    \node (B) [right of=A] {};
    \node (C) [below of=A] {};
    \node (D) [right of=C] {};
    
    \draw (A) to [out=135, in=-135, looseness=10] (A);
    \draw (A) to[bend left=20] (B);
    \draw (A) to[bend right=20] (B);
    \draw (B) to[bend left=20] (D);
    \draw (B) to[bend right=20] (D);
    \draw (D) to[bend left=20] (C);
    \draw (D) to[bend right=20] (C);
    \draw (C) to [out=135, in=-135, looseness=10] (C);
  \end{scope} 

  \begin{scope}[shift={(\hs,0)}] 
    \node (A) {};
    \node (B) [right of=A] {};
    \node (C) [below of=A] {};
    \node (D) [right of=C] {};
    
    \draw (A) to [out=135, in=-135, looseness=10] (A);
    \draw (A) to[bend left=20] (B);
    \draw (A) to[bend right=20] (B);
    \draw (B) to (D);
    \draw (B) to (C);
    \draw (D) to (C);
    \draw (D) to [out=45, in=-45, looseness=10] (D);
    \draw (C) to [out=135, in=-135, looseness=10] (C);
  \end{scope} 

  \begin{scope}[shift={(2*\hs,0)}] 
    \node (A) {};
    \node (B) [right of=A] {};
    \node (C) [below of=A] {};
    \node (D) [right of=C] {};
    
    \draw (A) to [out=135, in=-135, looseness=10] (A);
    \draw (A) to[bend left=20] (B);
    \draw (A) to[bend right=20] (B);
    \draw (B) to (D);
    \draw (B) to (C);
    \draw (D) to[bend left=20] (C);
    \draw (D) to[bend right=20] (C);
    \draw (D) to (C);
  \end{scope} 

  \begin{scope}[shift={(3*\hs,0)}] 
    \node (A) {};
    \node (B) [right of=A] {};
    \node (C) [below of=A] {};
    \node (D) [right of=C] {};
    
    \draw (A) -- (B);
    \draw (A) -- (C);
    \draw (B) -- (D);
    \draw (C) -- (D);
    \draw (A) to [out=135, in=-135, looseness=10] (A);
    \draw (B) to [out=45, in=-45, looseness=10] (B);
    \draw (C) to [out=135, in=-135, looseness=10] (C);
    \draw (D) to [out=45, in=-45, looseness=10] (D);    
  \end{scope} 

  \begin{scope}[shift={(4*\hs,0)}] 
    \node (A) {};
    \node (B) [right of=A] {};
    \node (C) [below of=A] {};
    \node (D) [right of=C] {};
    
    \draw (A) to [out=135, in=45, looseness=10] (A);
    \draw (B) to [out=45, in=135, looseness=10] (B);
    \draw (A) -- (B);
    \draw (A) -- (C);
    \draw (B) -- (D);
    \draw (C) -- (D);
    \draw (C) to [bend left=20] (D);
    \draw (C) to [bend right=20] (D);
  \end{scope} 

  \begin{scope}[shift={(0,-\vs)}] 
    \node (A) {};
    \node (B) [right of=A] {};
    \node (C) [below of=A] {};
    \node (D) [right of=C] {};

    \draw (A) to [out=135, in=-135, looseness=10] (A);
    \draw (A) -- (B);
    \draw (A) -- (C);
    \draw (B) to[bend left=20] (C);
    \draw (B) to[bend right=20] (C);
    \draw (B) -- (D);
    \draw (C) -- (D);
    \draw (D) to [out=45, in=-45, looseness=10] (D);
  \end{scope} 

  \begin{scope}[shift={(\hs,-\vs)}] 
    \node (A) {};
    \node (B) [right of=A] {};
    \node (C) [below of=A] {};
    \node (D) [right of=C] {};

    \draw (A) to [out=135, in=-135, looseness=10] (A);
    \draw (A) -- (B);
    \draw (A) -- (C);
    \draw (B) to (C);
    \draw (B) to[bend left=20] (D);
    \draw (B) to[bend right=20] (D);
    \draw (C) to[bend left=20] (D);
    \draw (C) to[bend right=20] (D);
  \end{scope} 

  \begin{scope}[shift={(2*\hs,-\vs)}] 
    \node (A) {};
    \node (B) [right of=A] {};
    \node (C) [below of=A] {};
    \node (D) [right of=C] {};

    \draw (A) to (B);
    \draw (A) to[bend left=20] (B);
    \draw (A) to[bend right=20] (B);
    \draw (A) -- (C);
    \draw (B) -- (D);
    \draw (C) -- (D);
    \draw (C) to[bend left=20] (D);
    \draw (C) to[bend right=20] (D);
  \end{scope} 

  \begin{scope}[shift={(3*\hs,-\vs)}] 
    \node (A) {};
    \node (B) [right of=A] {};
    \node (C) [below of=A] {};
    \node (D) [right of=C] {};

    \draw (A) to[bend left=20] (B);
    \draw (A) to[bend right=20] (B);
    \draw (A) to[bend left=20] (C);
    \draw (A) to[bend right=20] (C);
    \draw (B) to[bend left=20] (D);
    \draw (B) to[bend right=20] (D);
    \draw (C) to[bend left=20] (D);
    \draw (C) to[bend right=20] (D);
  \end{scope} 

  \begin{scope}[shift={(4*\hs,-\vs)}] 
    \node (A) {};
    \node (B) [right of=A] {};
    \node (C) [below of=A] {};
    \node (D) [right of=C] {};
    
    \draw (A) -- (B);
    \draw (A) -- (C);
    \draw (A) to[bend left=20] (D);
    \draw (A) to[bend right=20] (D);
    \draw (B) to[bend left=20] (C);
    \draw (B) to[bend right=20] (C);
    \draw (B) -- (D);
    \draw (C) -- (D);
  \end{scope} 

\end{tikzpicture}
\caption{One-Skeleta for Complexity 4}
\label{fig:c4}
\end{figure}
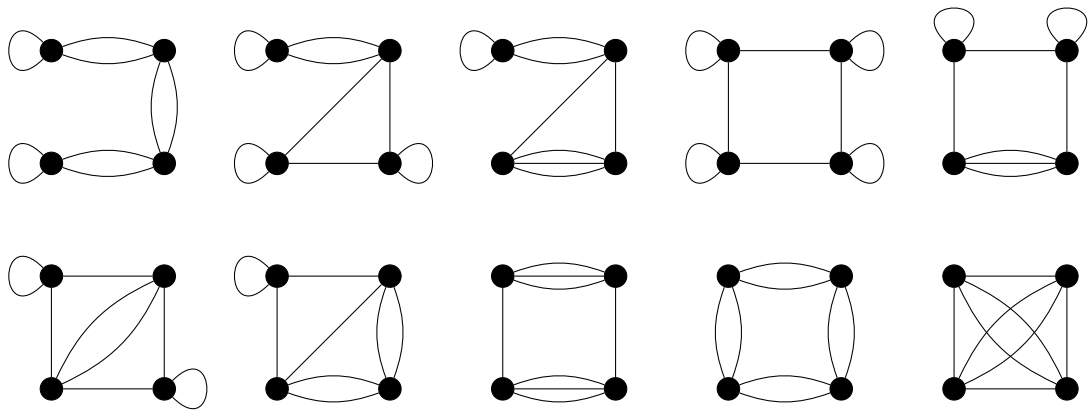
There are 4618 acyclic cellular fake surfaces of complexity 4. The 1-skeleton with the largest number of surfaces is $\Gamma^4_2$ with 1171. The one with the fewest is $\Gamma^4_9$ with 35. All of these are listed at \url{https://github.com/lucasfagan/Fake-Surfaces}.
\subsubsection{Complexity 5}
For complexity 5, there are 28 one-skeleta. We do not list the adjacency matrices, but the list can be found at \url{https://github.com/lucasfagan/Fake-Surfaces}. They are shown in order in Fig.~\ref{fig:c5real}. 

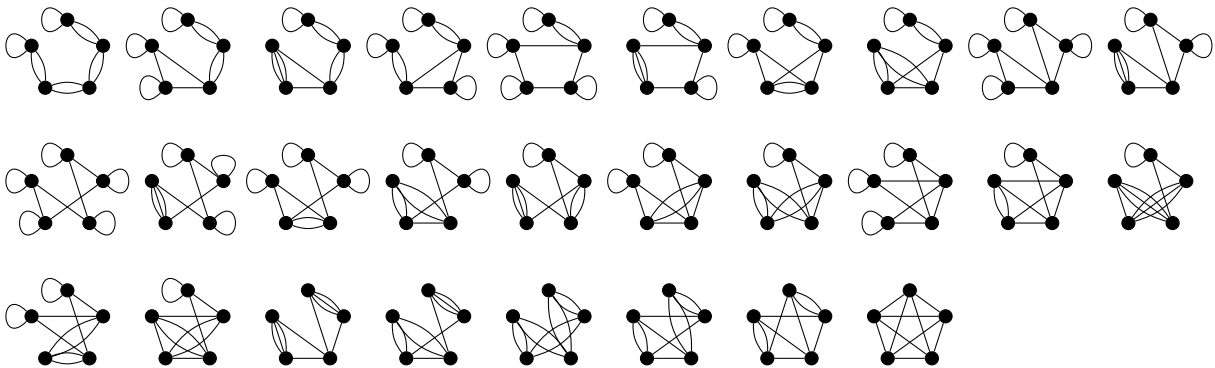
\begin{figure}[h]
\centering
\begin{tikzpicture}[every node/.style={circle, draw, fill=black, inner sep=1.75pt}]

  \def\hs{1.6} 
  \def\vs{1.8} 

  \pgfmathsetmacro\angle{360/5} 
  \pgfmathsetmacro\radius{0.5} 

  \newcommand{\drawpentagon}[2]{ 
    \begin{scope}[shift={(#1,#2)}] 
      \node (A) at (90:\radius) {};
      \node (B) at (90-\angle:\radius) {};
      \node (C) at (90-2*\angle:\radius) {};
      \node (D) at (90-3*\angle:\radius) {};
      \node (E) at (90-4*\angle:\radius) {};

\ifnum\row=3
\ifnum\col=1
\draw (A) to[out=135, in=-135, looseness=10] (A);
\draw (A) to[bend left=20] (B);
\draw (A) to[bend right=20] (B);
\draw (B) to[bend left=20] (C);
\draw (B) to[bend right=20] (C);
\draw (C) to[bend left=20] (D);
\draw (C) to[bend right=20] (D);
\draw (D) to[bend left=20] (E);
\draw (D) to[bend right=20] (E);
\draw (E) to[out=135, in=-135, looseness=10] (E);
\fi
\ifnum\col=2
\draw (A) to[out=135, in=-135, looseness=10] (A);
\draw (A) to[bend left=20] (B);
\draw (A) to[bend right=20] (B);
\draw (B) to[bend left=20] (C);
\draw (B) to[bend right=20] (C);
\draw (C) -- (D);
\draw (C) -- (E);
\draw (D) to[out=135, in=-135, looseness=10] (D);
\draw (D) -- (E);
\draw (E) to[out=135, in=-135, looseness=10] (E);
\fi
\ifnum\col=3
\draw (A) to[out=135, in=-135, looseness=10] (A);
\draw (A) to[bend left=20] (B);
\draw (A) to[bend right=20] (B);
\draw (B) to[bend left=20] (C);
\draw (B) to[bend right=20] (C);
\draw (C) -- (D);
\draw (C) -- (E);
\draw (D) to[bend left=20] (E);
\draw (D) to[bend right=20] (E);
\draw (D) -- (E);
\fi
\ifnum\col=4
\draw (A) to[out=135, in=-135, looseness=10] (A);
\draw (A) to[bend left=20] (B);
\draw (A) to[bend right=20] (B);
\draw (B) -- (C);
\draw (B) -- (D);
\draw (C) to[out=45, in=-45, looseness=10] (C);
\draw (C) -- (D);
\draw (D) to[bend left=20] (E);
\draw (D) to[bend right=20] (E);
\draw (E) to[out=135, in=-135, looseness=10] (E);
\fi
\ifnum\col=5
\draw (A) to[out=135, in=-135, looseness=10] (A);
\draw (A) to[bend left=20] (B);
\draw (A) to[bend right=20] (B);
\draw (B) -- (E);
\draw (B) -- (C);
\draw (C) to[out=45, in=-45, looseness=10] (C);
\draw (C) -- (D);
\draw (D) to[out=135, in=-135, looseness=10] (D);
\draw (D) -- (E);
\draw (E) to[out=135, in=-135, looseness=10] (E);
\fi
\ifnum\col=6
\draw (A) to[out=135, in=-135, looseness=10] (A);
\draw (A) to[bend left=20] (B);
\draw (A) to[bend right=20] (B);
\draw (B) -- (C);
\draw (B) -- (E);
\draw (C) to[out=45, in=-45, looseness=10] (C);
\draw (C) -- (D);
\draw (D) to[bend left=20] (E);
\draw (D) to[bend right=20] (E);
\draw (D) -- (E);
\fi
\ifnum\col=7
\draw (A) to[out=135, in=-135, looseness=10] (A);
\draw (A) to[bend left=20] (B);
\draw (A) to[bend right=20] (B);
\draw (B) -- (C);
\draw (B) -- (D);
\draw (C) to[bend left=20] (D);
\draw (C) to[bend right=20] (D);
\draw (C) -- (E);
\draw (D) -- (E);
\draw (E) to[out=135, in=-135, looseness=10] (E);
\fi
\ifnum\col=8
\draw (A) to[out=135, in=-135, looseness=10] (A);
\draw (A) to[bend left=20] (B);
\draw (A) to[bend right=20] (B);
\draw (B) -- (C);
\draw (B) -- (D);
\draw (C) -- (D);
\draw (C) to[bend left=20] (E);
\draw (C) to[bend right=20] (E);
\draw (D) to[bend left=20] (E);
\draw (D) to[bend right=20] (E);
\fi
\ifnum\col=9
\draw (A) to[out=135, in=-135, looseness=10] (A);
\draw (A) -- (B);
\draw (A) -- (C);
\draw (B) to[out=45, in=-45, looseness=10] (B);
\draw (B) -- (C);
\draw (C) -- (D);
\draw (C) -- (E);
\draw (D) to[out=135, in=-135, looseness=10] (D);
\draw (D) -- (E);
\draw (E) to[out=135, in=-135, looseness=10] (E);
\fi
\ifnum\col=10
\draw (A) to[out=135, in=-135, looseness=10] (A);
\draw (A) -- (B);
\draw (A) -- (C);
\draw (B) to[out=45, in=-45, looseness=10] (B);
\draw (B) -- (C);
\draw (C) -- (D);
\draw (C) -- (E);
\draw (D) to[bend left=20] (E);
\draw (D) to[bend right=20] (E);
\draw (D) -- (E);
\fi
\fi
\ifnum\row=2
\ifnum\col=1
\draw (A) to[out=135, in=-135, looseness=10] (A);
\draw (A) -- (B);
\draw (A) -- (C);
\draw (B) to[out=45, in=-45, looseness=10] (B);
\draw (B) -- (D);
\draw (C) to[out=45, in=-45, looseness=10] (C);
\draw (C) -- (E);
\draw (D) to[out=135, in=-135, looseness=10] (D);
\draw (D) -- (E);
\draw (E) to[out=135, in=-135, looseness=10] (E);
\fi
\ifnum\col=2
\draw (A) to[out=135, in=-135, looseness=10] (A);
\draw (A) -- (B);
\draw (A) -- (C);
\draw (B) to[out=45, in=135, looseness=10] (B);
\draw (B) -- (D);
\draw (C) to[out=45, in=-45, looseness=10] (C);
\draw (C) -- (E);
\draw (D) to[bend left=20] (E);
\draw (D) to[bend right=20] (E);
\draw (D) -- (E);
\fi
\ifnum\col=3
\draw (A) to[out=135, in=-135, looseness=10] (A);
\draw (A) -- (B);
\draw (A) -- (C);
\draw (B) to[out=45, in=-45, looseness=10] (B);
\draw (B) -- (D);
\draw (C) to[bend left=20] (D);
\draw (C) to[bend right=20] (D);
\draw (C) -- (E);
\draw (D) -- (E);
\draw (E) to[out=135, in=-135, looseness=10] (E);
\fi
\ifnum\col=4
\draw (A) to[out=135, in=-135, looseness=10] (A);
\draw (A) -- (B);
\draw (A) -- (C);
\draw (B) to[out=45, in=-45, looseness=10] (B);
\draw (B) -- (D);
\draw (C) -- (D);
\draw (C) to[bend left=20] (E);
\draw (C) to[bend right=20] (E);
\draw (D) to[bend left=20] (E);
\draw (D) to[bend right=20] (E);
\fi
\ifnum\col=5
\draw (A) to[out=135, in=-135, looseness=10] (A);
\draw (A) -- (B);
\draw (A) -- (C);
\draw (B) to[bend left=20] (C);
\draw (B) to[bend right=20] (C);
\draw (B) -- (D);
\draw (C) -- (E);
\draw (D) to[bend left=20] (E);
\draw (D) to[bend right=20] (E);
\draw (D) -- (E);
\fi
\ifnum\col=6
\draw (A) to[out=135, in=-135, looseness=10] (A);
\draw (A) -- (B);
\draw (A) -- (C);
\draw (B) -- (C);
\draw (B) to[bend left=20] (D);
\draw (B) to[bend right=20] (D);
\draw (C) -- (D);
\draw (C) -- (E);
\draw (D) -- (E);
\draw (E) to[out=135, in=-135, looseness=10] (E);
\fi
\ifnum\col=7
\draw (A) to[out=135, in=-135, looseness=10] (A);
\draw (A) -- (B);
\draw (A) -- (C);
\draw (B) -- (C);
\draw (B) to[bend left=20] (D);
\draw (B) to[bend right=20] (D);
\draw (C) to[bend left=20] (E);
\draw (C) to[bend right=20] (E);
\draw (D) to[bend left=20] (E);
\draw (D) to[bend right=20] (E);
\fi
\ifnum\col=8
\draw (A) to[out=135, in=-135, looseness=10] (A);
\draw (A) -- (B);
\draw (A) -- (C);
\draw (B) -- (C);
\draw (B) -- (D);
\draw (B) -- (E);
\draw (C) -- (D);
\draw (C) -- (E);
\draw (D) to[out=135, in=-135, looseness=10] (D);
\draw (E) to[out=135, in=-135, looseness=10] (E);
\fi
\ifnum\col=9
\draw (A) to[out=135, in=-135, looseness=10] (A);
\draw (A) -- (B);
\draw (A) -- (C);
\draw (B) -- (C);
\draw (B) -- (D);
\draw (B) -- (E);
\draw (C) -- (D);
\draw (C) -- (E);
\draw (D) to[bend left=20] (E);
\draw (D) to[bend right=20] (E);
\fi
\ifnum\col=10
\draw (A) to[out=135, in=-135, looseness=10] (A);
\draw (A) -- (B);
\draw (A) -- (C);
\draw (B) to[bend left=20] (D);
\draw (B) to[bend right=20] (D);
\draw (B) -- (D);
\draw (C) to[bend left=20] (E);
\draw (C) to[bend right=20] (E);
\draw (C) -- (E);
\draw (D) -- (E);
\fi
\fi
\ifnum\row=1
\ifnum\col=1
\draw (A) to[out=135, in=-135, looseness=10] (A);
\draw (A) -- (B);
\draw (A) -- (C);
\draw (B) to[bend left=20] (D);
\draw (B) to[bend right=20] (D);
\draw (B) -- (E);
\draw (C) to[bend left=20] (D);
\draw (C) to[bend right=20] (D);
\draw (C) -- (E);
\draw (E) to[out=135, in=-135, looseness=10] (E);
\fi
\ifnum\col=2
\draw (A) to[out=135, in=-135, looseness=10] (A);
\draw (A) -- (B);
\draw (A) -- (C);
\draw (B) to[bend left=20] (D);
\draw (B) to[bend right=20] (D);
\draw (B) -- (E);
\draw (C) -- (D);
\draw (C) to[bend left=20] (E);
\draw (C) to[bend right=20] (E);
\draw (D) -- (E);
\fi
\ifnum\col=3
\draw (A) to[bend left=20] (B);
\draw (A) to[bend right=20] (B);
\draw (A) -- (B);
\draw (A) -- (C);
\draw (B) -- (C);
\draw (C) -- (D);
\draw (C) -- (E);
\draw (D) to[bend left=20] (E);
\draw (D) to[bend right=20] (E);
\draw (D) -- (E);
\fi
\ifnum\col=4
\draw (A) to[bend left=20] (B);
\draw (A) to[bend right=20] (B);
\draw (A) -- (B);
\draw (A) -- (C);
\draw (B) -- (D);
\draw (C) -- (D);
\draw (C) to[bend left=20] (E);
\draw (C) to[bend right=20] (E);
\draw (D) to[bend left=20] (E);
\draw (D) to[bend right=20] (E);
\fi
\ifnum\col=5
\draw (A) to[bend left=20] (B);
\draw (A) to[bend right=20] (B);
\draw (A) to[bend left=20] (C);
\draw (A) to[bend right=20] (C);
\draw (B) to[bend left=20] (D);
\draw (B) to[bend right=20] (D);
\draw (C) to[bend left=20] (E);
\draw (C) to[bend right=20] (E);
\draw (D) to[bend left=20] (E);
\draw (D) to[bend right=20] (E);
\fi
\ifnum\col=6
\draw (A) to[bend left=20] (B);
\draw (A) to[bend right=20] (B);
\draw (A) to[bend left=20] (C);
\draw (A) to[bend right=20] (C);
\draw (B) -- (D);
\draw (B) -- (E);
\draw (C) -- (D);
\draw (C) -- (E);
\draw (D) to[bend left=20] (E);
\draw (D) to[bend right=20] (E);
\fi
\ifnum\col=7
\draw (A) to[bend left=20] (B);
\draw (A) to[bend right=20] (B);
\draw (A) -- (C);
\draw (A) -- (D);
\draw (B) -- (C);
\draw (B) -- (E);
\draw (C) -- (D);
\draw (C) -- (E);
\draw (D) to[bend left=20] (E);
\draw (D) to[bend right=20] (E);
\fi
\ifnum\col=8
\draw (A) -- (B);
\draw (A) -- (C);
\draw (A) -- (D);
\draw (A) -- (E);
\draw (B) -- (C);
\draw (B) -- (D);
\draw (B) -- (E);
\draw (C) -- (D);
\draw (C) -- (E);
\draw (D) -- (E);
\fi
\fi
    \end{scope}
  }

  \foreach \row in {2,...,3} { 
    \foreach \col in {1,...,10} {
      \drawpentagon{\col*\hs}{\row*\vs}
    }
  }
\foreach \row in {1,...,1} {
  \foreach \col in {1,...,8} {
      \drawpentagon{\col*\hs}{\row*\vs}
      }
  }

\end{tikzpicture}
\caption{One-Skeleta for Complexity 5}
\label{fig:c5real}
\end{figure}

Important to note is that complexity 5 is the first in which there exists a 1-skeleton that is a simple graph, i.e., does not have a double edge or self loop: $K_5$. 

Due to computational limitations, we chose to classify only those surfaces where there are no small disks. Disks of boundary length 1 and 2 are necessarily embedded in complexity greater than 1, so this still allowed us to check that all complexity 5 surfaces have embedded disks. All acyclic cellular fake surfaces of complexity 5 without small disks are listed at \url{https://github.com/lucasfagan/Fake-Surfaces}. 

\section{Contractibility and Embedded Disk Conjectures}
\label{sec:conjs}

\subsection{Contractability conjecture}

Ikeda proved that any acyclic cellular fake surface spine of complexity less than 3 is contractible.  In Conjecture 2 on Page 4 of \cite{robertson1973fake}, it is shown that any acyclic cellular fake surface of complexity less than 4 is contractible, and it is conjectured that the same holds for complexity 4.  In this section, we prove this conjecture.  This is the best possible general result as the dodecahedral Poincar\'e homology sphere has a spine of complexity 5 that is not contractible.

\begin{restated-cor}
    Acyclic cellular fake surfaces of complexity 1, 2, 3 and 4 are all contractible.
\end{restated-cor}

The proof is by an explicit computer check using our classification.  The computer code to check that an acyclic surface is contractible can be found in Appendix~\ref{sec:code} and at \url{https://github.com/lucasfagan/Fake-Surfaces}.  For complexity 1 to 4, we verified that the fundamental groups are all trivial.

\subsection{Embedded disk conjecture}

Embedded disks are useful for the proof of the Zeeman conjecture.  For complexity 1 contractible fake surfaces, the Zeeman conjecture is proved using the embedded disks \cite{ikeda1971acyclic}.  In \cite{fagan-qiu-wang-stableac}, we extend the proof of Zeeman conjecture to complexity up to 5 for contractible fake surfaces.  Our proof relies on the following:

\begin{restated-thm}
    All contractible fake surfaces up to complexity 5 have embedded disks.
\end{restated-thm}

The statement for complexity up to 4 is verified explicitly using our classification. We do not have a complete list of all complexity 5 acyclic cellular fake surfaces as in complexity up to 4.
But we have such a list for all acyclic cellular fake surfaces without small disks.
All those fake surfaces have embedded disks as small disks are always embedded.

We conjecture that all contractible fake surfaces have embedded disks.  The embedded disks can be thought as the weak parts of the resulting presentations of the trivial group.

\section*{Acknowledgments}
Z.W. is partially supported by NSF grant CCF 2006463 and ARO MURI contract W911NF-20-1-0082. Use was made of computational facilities purchased with funds from the National Science Foundation (CNS-1725797, OAC-1925717) and administered by the Center for Scientific Computing (CSC). The CSC is supported by the California NanoSystems Institute and the Materials Research Science and Engineering Center (MRSEC; NSF DMR 2308708) at UC Santa Barbara.

The authors report there are no competing interests to declare. 
\bibliographystyle{abbrv}
\bibliography{references}

@misc{cfs,
author = {{OEIS Foundation Inc.}},
title = {Entry {A}373555 in The {O}n-{L}ine {E}ncyclopedia of {I}nteger {S}equences},
year = {2024},
howpublished = {\url{https://oeis.org/A373555}},
}

@article{costantino2004short,
  title={A short introduction to shadows of 4-manifolds},
  author={Costantino, Francesco},
  journal={arXiv preprint math/0405582},
  year={2004}
}

@article{fagan-qiu-wang-stableac,
  title={Stable Andrews-Curtis Conjecture via Fake Surfaces and Zeeman Conjecture},
  author={Fagan, Lucas and Qiu, Yang and Wang, Zhenghan},
  journal={arXiv preprint math/2412.12293},
  year={2024}
}

@article{taylor1976structure,
  title={The structure of singularities in soap-bubble-like and soap-film-like minimal surfaces},
  author={Taylor, Jean E},
  journal={Annals of Mathematics},
  volume={103},
  number={3},
  pages={489--539},
  year={1976},
  publisher={JSTOR}
}

@book{hog1993two,
  title={Two-dimensional homotopy and combinatorial group theory},
  author={Hog-Angeloni, Cynthia and Metzler, Wolfgang and Sieradski, Allan J},
  volume={197},
  year={1993},
  publisher={Cambridge University Press}
}

@article{casler1965imbedding,
  title={An imbedding theorem for connected 3-manifolds with boundary},
  author={Casler, Burtis G},
  journal={Proceedings of the American Mathematical Society},
  volume={16},
  number={4},
  pages={559--566},
  year={1965}
}

@article{turaev1992state,
  title={State sum invariants of 3-manifolds and quantum 6j-symbols},
  author={Turaev, Vladimir G and Viro, Oleg Ya},
  journal={Topology},
  volume={31},
  number={4},
  pages={865--902},
  year={1992},
  publisher={Pergamon}
}

@book{turaev2010quantum,
  title={Quantum invariants of knots and 3-manifolds},
  author={Turaev, Vladimir G},
  year={2010},
  publisher={de Gruyter}
}

@article{martelli2011four,
  title={Four-manifolds with shadow-complexity zero},
  author={Martelli, Bruno},
  journal={International Mathematics Research Notices},
  volume={2011},
  number={6},
  pages={1268--1351},
  year={2011},
  publisher={OUP}
}

@inproceedings{koda2022four,
  title={Four-manifolds with shadow-complexity one},
  author={Koda, Yuya and Martelli, Bruno and Naoe, Hironobu},
  booktitle={Annales de la Facult{\'e} des sciences de Toulouse: Math{\'e}matiques},
  volume={31},
  number={4},
  pages={1111--1212},
  year={2022}
}

@article{hironobu2018mazur,
  title={MAZUR MANIFOLDS AND CORKS WITH SMALL SHADOW COMPLEXITIES},
  author={Hironobu, Naoe},
  journal={Osaka Journal of Mathematics},
  volume={55},
  number={3},
  pages={479--498},
  year={2018},
  publisher={Osaka University and Osaka City University, Departments of Mathematics}
}

@book{freedman2014topology,
  title={Topology of 4-Manifolds (PMS-39), Volume 39},
  author={Freedman, Michael H and Quinn, Frank},
  volume={46},
  year={2014},
  publisher={Princeton University Press}
}

@article{ikeda1971acyclic,
  title={Acyclic fake surfaces},
  author={Ikeda, Hiroshi},
  journal={Topology},
  volume={10},
  number={1},
  pages={9--36},
  year={1971},
  publisher={Pergamon}
}

@article{matveev1973special,
  title={Special spines of piecewise linear manifolds},
  author={Matveev, Sergei Vladimirovich},
  journal={Mathematics of the USSR-Sbornik},
  volume={21},
  number={2},
  pages={279},
  year={1973},
  publisher={IOP Publishing}
}

@phdthesis{robertson1973fake,
  author       = {Robertson, Jimmie Arnold},
  title        = {Fake-Surfaces which are Spines of 3-Balls},
  school  = {The Florida State University},
  department   = {Department of Mathematics}, 
  type         = {Ph.{D}. thesis},
  year         = {1973},
  address      = {Tallahassee, FL},     
}

@techreport{zeeman1963seminar,
  author       = {Zeeman, Erik Christopher},
  title        = {Seminar on Combinatorial Topology},
  institution  = {Institut des Hautes Études Scientifiques},
  year         = {1963},
  address      = {Bures-sur-Yvette, France},
  type         = {Seminar notes},
  note         = {},
}

@article{wright1973formal,
  title={Formal 3-deformations of 2-polyhedra},
  author={Wright, Perrin},
  journal={Proceedings of the American Mathematical Society},
  volume={37},
  number={1},
  pages={305--308},
  year={1973}
}

@book{matveev2007algorithmic,
  title={Algorithmic topology and classification of 3-manifolds},
  author={Matveev, Sergei},
  volume={9},
  year={2007},
  publisher={Springer}
}

@book{rourke1982introduction,
  title={Introduction to Piecewise-linear Topology},
  author={Rourke, C.P. and Sanderson, B.J.},
  isbn={9783540111023},
  lccn={81018311},
  series={Ergebnisse der Mathematik und ihrer Grenzgebiete},
  url={https://books.google.com/books?id=AoEpAQAAMAAJ},
  year={1982},
  publisher={Springer-Verlag}
}

@article{koda2020shadows,
  title={Shadows of acyclic 4--manifolds with sphere boundary},
  author={Koda, Yuya and Naoe, Hironobu},
  journal={Algebraic \& Geometric Topology},
  volume={20},
  number={7},
  pages={3707--3731},
  year={2020},
  publisher={Mathematical Sciences Publishers}
}

\appendix
\section{Code to Check Contractibility of Acyclic Fake Surface}
\label{sec:code}
In order to check if an acyclic fake surface is contractible, it suffices to check that it is simply connected by the Whitehead theorem and the Hurewicz theorem. We provide a Sage code snippet that takes as input a fake surface and a choice of maximal tree, and outputs if the fundamental group is trivial. Of course, deciding if a group presentation is the trivial group is undecideable, but this only leaves open the possibility of false negatives. Sage is able to detect the triviality of the fundamental group of every fake surface through complexity 5, with the lone exception being the known counterexample of the complexity 5 spine of the Poincar\'e homology sphere. 

The code can also be found at \url{https://github.com/lucasfagan/Fake-Surfaces}.

\begin{lstlisting}[language=Python, style=mystyle, breaklines=true]
# sample data from complexity 2 to demonstrate formatting
acyclic_surface = [[4,2,-1,-1,-2,3,2,1,-3],[2,-3],[4]]
maximal_tree = [2]

def is_contractible(acyclic_surface, maximal_tree):
    # remove the elements in the maximal tree
    collapsed_surface = [[edge for edge in disk if (edge not in maximal_tree and -1*edge not in maximal_tree)] for disk in acyclic_surface]
    # for conveinence, reduce the remaining edges to be labeled with 1 through n+1
    for removed_edge in sorted(maximal_tree, reverse=True):
        collapsed_surface = [[edge-1 if edge>removed_edge else edge for edge in disk] for disk in collapsed_surface]
        collapsed_surface = [[edge+1 if edge<-1*removed_edge else edge for edge in disk] for disk in collapsed_surface] 

    # create a free group with one generator for each remaining edge
    F = FreeGroup(len(maximal_tree)+2,'a')

    group_relations = []
    
    # convert the disks to relations 
    for disk in collapsed_surface:
        relation = F.one()  # Initialize with identity
        for edge in disk:
            if edge < 0:
                relation *= (F.gen(-edge-1))**-1  # Inverse of generator
            else:
                relation *= F.gen(edge-1)  # Multiply with generator

        # Append the relation
        group_relations.append(relation)

    # Create quotient group
    G = F.quotient(group_relations) 
    # Simplify the group
    G = G.simplified()

    # Check if group is trivial
    if G.order()!=1:
        return False
    return True

print(is_contractible(acyclic_surface, maximal_tree))
\end{lstlisting}

Saved as a Sage file, we can simply run this with the command
\begin{verbatim}
    sage filename.sage
\end{verbatim}
\section{Fake Surfaces of Complexity 2}
All acyclic cellular fake surfaces of complexity 2 are listed below. The rest of the data can be found at \url{https://github.com/lucasfagan/Fake-Surfaces}.
\label{sec:fs-c23}
\begin{multicols}{3}
{\fontsize{8pt}{12pt}\selectfont 
\[\left(\Gamma_1^2,\left[\begin{array}{c|cccc|cccc|cc}
42\text{-}1\text{-}1\text{-}2431\text{-}3&NN\\
2\text{-}3&YN\\
4&YN\end{array}\right]\right)\]
\[\left(\Gamma_1^2,\left[\begin{array}{c|cccc|cccc|cc}
43\text{-}1\text{-}1\text{-}2\text{-}431\text{-}2&NY\\
2\text{-}3&YN\\
4&YN\end{array}\right]\right)\]
\[\left(\Gamma_1^2,\left[\begin{array}{c|cccc|ccccc|cc}
43\text{-}2\text{-}42\text{-}1\text{-}231\text{-}3&NY\\
4&YY\\
\text{-}1&YY\end{array}\right]\right)\]
\[\left(\Gamma_1^2,\left[\begin{array}{c|cccc|ccccc|cc}
431\text{-}23\text{-}1\text{-}242\text{-}3&NN\\
4&YY\\
\text{-}1&YN\end{array}\right]\right)\]
\[\left(\Gamma_1^2,\left[\begin{array}{c|cccc|ccccc|cc}
42\text{-}1\text{-}3\text{-}43\text{-}23\text{-}1\text{-}2&NY\\
4&YY\\
\text{-}1&YN\end{array}\right]\right)\]
\[\left(\Gamma_1^2,\left[\begin{array}{c|cccc|ccccc|cc}
3\text{-}1\text{-}32\text{-}1\text{-}2\text{-}42\text{-}3\text{-}4&NY\\
4&YY\\
\text{-}1&YY\end{array}\right]\right)\]
\[\left(\Gamma_1^2,\left[\begin{array}{c|ccccc|cccc|cc}
42\text{-}1\text{-}231\text{-}3\text{-}42\text{-}3&NY\\
4&YN\\
\text{-}1&YY\end{array}\right]\right)\]
\[\left(\Gamma_1^2,\left[\begin{array}{c|ccccc|cccc|cc}
3\text{-}23\text{-}1\text{-}3\text{-}42\text{-}1\text{-}2\text{-}4&NN\\
4&YN\\
\text{-}1&YY\end{array}\right]\right)\]
\[\left(\Gamma_1^2,\left[\begin{array}{c|c|cccc|cccc|cc}
42\text{-}3431\text{-}32\text{-}1\text{-}2&NN\\
4&YY\\
\text{-}1&YY\end{array}\right]\right)\]
\[\left(\Gamma_1^2,\left[\begin{array}{c|cccccc|cc}
442\text{-}32\text{-}1\text{-}2&NN\\
43\text{-}1\text{-}3&NN\\
\text{-}1&YY\end{array}\right]\right)\]
\[\left(\Gamma_1^2,\left[\begin{array}{c|cccccc|cc}
3\text{-}1\text{-}3\text{-}42\text{-}1\text{-}2&NY\\
442\text{-}3&NY\\
\text{-}1&YY\end{array}\right]\right)\]
\[\left(\Gamma_1^2,\left[\begin{array}{c|cccccc|cc}
43\text{-}1\text{-}1\text{-}32\text{-}3&NY\\
2\text{-}1\text{-}2\text{-}4&NN\\
4&YY\end{array}\right]\right)\]
\[\left(\Gamma_1^2,\left[\begin{array}{c|cccccc|cc}
43\text{-}1\text{-}2\text{-}42\text{-}3&NY\\
2\text{-}1\text{-}1\text{-}3&NN\\
4&YY\end{array}\right]\right)\]
\[\left(\Gamma_1^2,\left[\begin{array}{c|cccc|ccc|cc}
21\text{-}3\text{-}42\text{-}1\text{-}1\text{-}3&NN\\
42\text{-}3&NY\\
4&YN\end{array}\right]\right)\]
\[\left(\Gamma_1^2,\left[\begin{array}{c|cccc|ccc|cc}
443\text{-}1\text{-}2\text{-}42\text{-}3&NY\\
2\text{-}1\text{-}3&NY\\
\text{-}1&YN\end{array}\right]\right)\]

\[\left(\Gamma_2^2,\left[\begin{array}{cccccccc|cc}
4\text{-}31\text{-}42\text{-}41\text{-}3&NN\\
1\text{-}2&YN\\
2\text{-}3&YN\end{array}\right]\right)\]
\[\left(\Gamma_2^2,\left[\begin{array}{cccccccc|cc}
4\text{-}12\text{-}43\text{-}41\text{-}2&NY\\
2\text{-}3&YN\\
1\text{-}3&YN\end{array}\right]\right)\]

}
\end{multicols}
\end{document}